\documentclass{ws-rv9x6}
\usepackage{subfigure}   
\usepackage{ws-rv-thm}   
\usepackage{ws-rv-van}   
\usepackage{mathrsfs}
\usepackage{tikz}

\newcommand\bdm{\begin{proof}}
\newcommand\edm{\end{proof}}

\makeindex

\begin{document}

\chapter[The orthogonality of Al-Salam-Carlitz polynomials for complex parameters]
{The orthogonality of Al-Salam-Carlitz polynomials for complex parameters
}\label{ra_ch1}

\author[H.~S.~Cohl, R.~S.~Costas-Santos and W.~Xu]{
Howard S. Cohl${}^\ast$, Roberto S. Costas-Santos$^{\dag}$ and Wenqing Xu$^{\ddag}$
}

\address{${}^{\ast}$Applied and Computational Mathematics Division, \\
National Institute of Standards and Technology, \\
Gaithersburg, MD 20899-8910, USA\\
{\tt howard.cohl@nist.gov} \\[0.2cm] 
${}^\dag$Departamento de F\'isica y Matem\'aticas, Facultad de Ciencias, Universidad de
Alcal\'a, 28871 Alcal\'a de Henares, Madrid, Spain\\
{\tt rscosa@gmail.com}\\[0.2cm]
${}^\ddag$Department of Mathematics and Statistics, California
Institute of Technology, CA 91125, USA\\
{\tt williamxuxu@yahoo.com}
}

\begin{abstract}
In this contribution, we study the orthogonality conditions satisfied by
Al-Salam-Carlitz polynomials $U^{(a)}_n(x;q)$ when the parameters
$a$ and $q$ are not necessarily real nor `classical', i.e., the linear functional
$\bf u$ with respect to such polynomial sequence is quasi-definite and not
positive definite.
We establish orthogonality on a simple contour in the complex plane
which depends on the parameters. In all cases we show that the orthogonality
conditions characterize the Al-Salam-Carlitz polynomials $U_n^{(a)}(x;q)$
of degree $n$ up to a constant factor.
We also obtain a generalization of the unique generating function for
these polynomials.
\end{abstract}
\body

\vspace{0.2cm}
\noindent Keywords:~$q$-orthogonal polynomials; $q$-difference operator; $q$-integral
representation; discrete measure.\\
MSC classification:~33C45; 42C05\\[-0.5cm]


\section{Introduction}

\noindent The Al-Salam-Carlitz polynomials $U_n^{(a)}(x;q)$ were introduced by 
W.~A.~Al-Salam and L. Carlitz in \cite{AlSaCa} as follows:
\begin{equation}\label{1:1}
U_n^{(a)}(x;q):=(-a)^n q^{n\choose 2}\sum_{k=0}^n \frac{(q^{-n};q)_k 
(x^{-1};q)_k}{(q;q)_k} \frac{q^k x^k}{a^k}.
\end{equation}
In fact, these polynomials have a Rodrigues-type formula 
\cite[(3.24.10)]{Koekoeketal}
\[
U_n^{(a)}(x;q)=\frac{a^n q^{n \choose 2}(1-q)^n}{q^n  w(x;a;q)}
{\mathscr D}^n_{q^{-1}}\big( w(x;a;q)\big),
\]
where
\[
w(x;a;q):=(qx;q)_\infty(qx/a;q)_\infty,
\]
the $q$-Pochhammer symbol ($q$-shifted factorial) is defined as
\[
(z;q)_0:=1,\quad (z;q)_n:=\prod_{k=0}^{n-1} (1-zq^k),
\]
\[
(z;q)_\infty:=\prod_{k=0}^{\infty} (1-zq^k), \quad |z|<1,
\]
and the $q$-derivative operator is
defined by
\[
{\mathscr D}_q f(z):=\left\{\begin{array}{cl}
\dfrac{f(qz)-f(z)}{(q-1)z} & \text{if} \ q\neq 1 \ {\rm and} \ z\ne 0, \\[4mm]
f'(z) & \text{if}  \  q=1 \ {\rm or} \ z=0.\end{array}\right.
\]
\begin{remark}
Observe that by the definition of the $q$-derivative
\[
{\mathscr D}_{q^{-1}} f(z)={\mathscr D}_{q} f(qz),\quad 
{\rm and} \quad 
{\mathscr D}^n_{q^{-1}} f(z):={\mathscr D}^{n-1}_{q^{-1}} \big({\mathscr D}_{q^{-1}} f(z)\big),
\ n=2, 3, \dots
\]
\end{remark}
The expression (\ref{1:1}) shows us that $U_n^{(a)}(x;q)$ is an analytic 
function for any complex value parameters $a$ and $q$, and thus can 
be considered for general $a, q\in \mathbb C\setminus \{0\}$.

The classical Al-Salam-Carlitz polynomials correspond to parameters 
$a<0$ and $0<q<1$. For these parameters, the Al-Salam-Carlitz
polynomials are orthogonal on $[a, 1]$ with respect to the weight
function $w$.
More specifically,  
for $a<0$ and $0<q<1$ \cite[(14.24.2)]{Koekoeketal},
\[
\int_a^1 U_n^{(a)}(x;q)U_m^{(a)}(x;q) (qx,qx/a;q)_\infty d_q x
= d_n^2 \delta_{n,m},
\]
where 
\[
d_n^2:=
(-a)^n (1-q)(q;q)_n (q;q)_\infty (a;q)_\infty (q/a;q)_\infty
q^{n\choose 2},
\]
and the $q$-Jackson integral \cite[(1.15.7)]{Koekoeketal} is defined as
\[
\int_a^b f(x)d_q x:=\int_0^b f(x)d_q x-\int_0^a f(x)d_q x,
\]
where
\[
\int_0^a f(x)d_q x:=a(1-q)\sum_{n=0}^\infty f(aq^n)q^n.
\]
Taking into account the previous orthogonality relation, it is a direct
result that if $a$ and $q$ are classical, i.e., $a$, $q\in \mathbb R$, with 
$a\ne 1$, $0<q<1$ all the  zeros of $U_n^{(a)}(x;q)$ are simple and belong 
to the interval $[a,1]$, but this is no longer valid  for general $a$ and $q$ complex.  
In this paper we show that for general $a$, $q$ complex numbers,
but excluding some special cases, the Al-Salam-Carlitz polynomials
$U_n^{(a)}(x;q)$ may still be characterized by orthogonality relations.
The case $a<0$ and $0<q<1$ or $0<aq<1$ and $q>1$ are classical, i.e., 
the linear functional $\bf u$ with respect to such polynomial sequence is 
orthogonal is positive definite and in such a case there exists a weight function 
$\omega(x)$ so that 
\[
\langle {\bf u}, p\rangle=\int_{a}^1 p(x)\, \omega(x)\, dx,\quad p\in \mathbb P[x].
\]
Note that this is the key for the study of many properties of Al-Salam-Carlitz
polynomials I and  II.  Thus, our goal is to establish orthogonality
conditions  for most of the remaining cases for which 
the linear form $\bf u$ is quasi-definite, i.e., for all $n, m\in \mathbb N_0$
\[
\langle {\bf u}, p_n p_m \rangle=k_n \delta_{n,m},\quad k_n\ne 0.
\]
We believe that these new orthogonality conditions can be useful 
in the study of the zeros of Al-Salam-Carlitz polynomials. For general 
$a, q\in \mathbb C\setminus\{0\}$,  the zeros are not confined to a 
real interval, but they distribute themselves in the complex plane
as we can see in Figure 1. 
Throughout this paper denote $p:=q^{-1}$.
\vspace{-0.2cm}
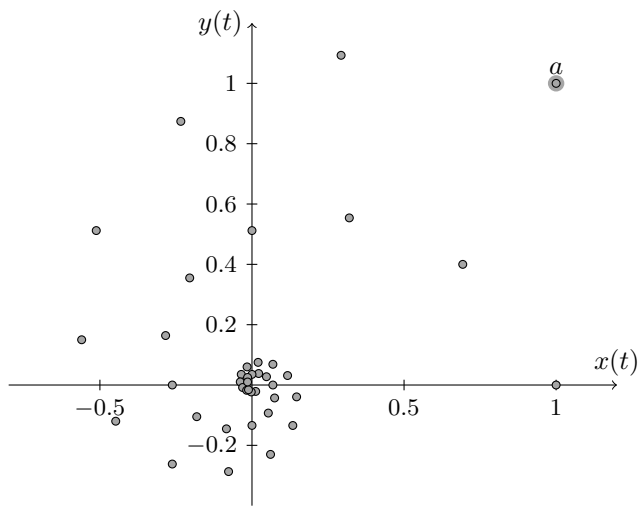
\begin{figure}[hbtp!]
\label{fig1}
\begin{center}
\begin{tikzpicture}[domain=-0.8:1.3,scale=4]
\draw[->] (-0.8,0) -- (1.2,0) node[above] {$x(t)$};
\draw[->] (0,-0.4) -- (0,1.2) node[left] {$y(t)$};
\foreach \x/\xtext in {1/1, 0.5/0.5,-0.5/-0.5}
\draw[shift={(\x,0)}] (0pt,0.5pt) -- (0pt,-0.5pt) node[below] {\small $\xtext$};
\foreach \y/\ytext in {-0.2/-0.2, 0.2/0.2,0.4/0.4,0.6/0.6,0.8/0.8,1/1}
\draw[shift={(0,\y)}] (0.5pt,0pt) -- (-0.5pt,0pt) node[left] {\small $\ytext$};
\fill[gray!70] (1,1) circle (0.027) node[above,black]{$a$};
\draw plot[only marks, mark=*, mark options={fill=gray!70},mark
size=0.37pt] coordinates{(1., 1.)(0.293, 1.093)(1., 0) (-0.234,  0.874) 
(0.693, 0.4) (-0.512, 0.512)(0.32,  0.554)(-0.56, 0.15)(0,  0.512)
(-0.448, -0.12)(-0.2048,  0.355)(-0.262, -0.262) (-0.284,   0.164)
(-0.077, -0.287)(-0.262, 0) (0.061, -0.23) (-0.1816, -0.105) (0.1342, -0.1342)
(-0.084, -0.1453) (0.1467, -0.0393) (0, -0.134) (0.117, 0.03144)(0.0537, -0.093)
(0.0687195, 0.0687195) (0.074391, -0.0429497)(0.0201225,   0.075098)
(0.0687195, 0)(-0.016098, 0.0600784)(0.0476103, 0.0274878)
(-0.0351844, 0.0351844) (0.0219907, 0.038088)(-0.0384484,   0.010304)
(0.0000322943, 0.0351876) (-0.0307947, -0.00822027) (-0.0144521,   0.0250943)
(-0.0175046, -0.017489)(0.0121248, -0.0212596)
(-0.00238931, -0.0230145)(-0.0114884, -0.0157088)(-0.0145,  0.01)};
\end{tikzpicture}
\end{center}
\caption{Zeros of $U^{(1+i)}_{30}\left(x;\frac45 \exp(\pi i/6)\right)$}
\end{figure}
\vspace{0.1cm}
\section{Orthogonality in the complex plane}
\label{sectionorthog}
\begin{theorem} \label{thm:3.1}
Let $a, q\in \mathbb C$, $a\ne 0, 1$,  $0<|q|<1$, the Al-Salam-Carlitz  
polynomials are  the  unique polynomials (up to a multiplicative constant) satisfying 
the property of orthogonality
\begin{equation} \label{2:1}
\int_a^1 U_n^{(a)}(x;q)U_m^{(a)}(x;q)w(x;a;q)d_q x=d_n^2 
\delta_{n,m}.
\end{equation}
\end{theorem}
\begin{remark}
I if $0<|q|<1,$ the lattice $\{q^k:k\in \mathbb N_0\}\cup
\{aq^k:k\in \mathbb N_0\}$ is a set of points which are located inside on a single 
contour that goes from 1 to 0, and then from 0 to $a,$ through the spirals 
\[
S_1: z(t)=|q|^t \exp(it\arg q),\quad S_2: z(t)=|a||q|^t \exp(it\arg q+i\arg a), 
\]
where $0<|q|<1,$ $t\in [0,\infty)$, which 
we can see in Figure 2. 
Taking into account \eqref{2:1}, we need to avoid the $a=1$ case. For the $a=0$ case, 
we cannot apply Favard's result \cite{chi1}, because in such a 
case this polynomial sequence fulfills the recurrence relation \cite{Koekoeketal}
\[
U_{n+1}^{(0)}(x;q)=(x-q^n)U_n^{(0)}(x;q),\quad
U_0^{(0)}(x;q)=1.
\]
\end{remark}
\begin{figure}[!hbtp]
\label{fig2}
\begin{center}
\includegraphics[scale=0.65]{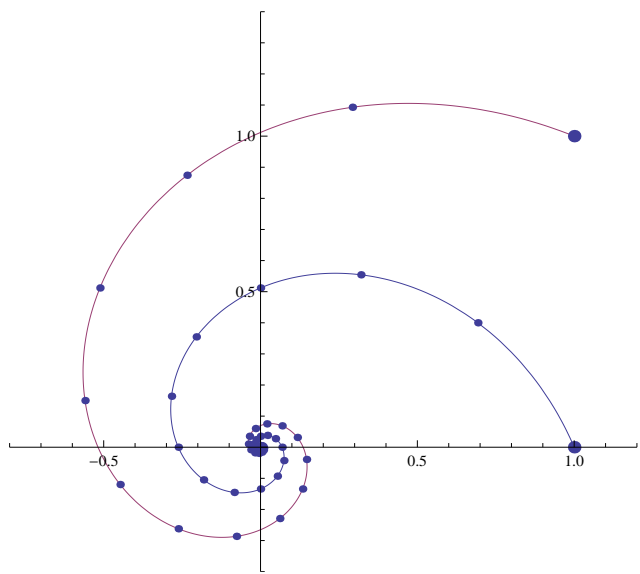}

\end{center}
\caption{
The lattice $\{q^k:k\in \mathbb N_0\}\cup
\{(1+i)q^k:k\in \mathbb N_0\}$ with $q=4/5\exp(\pi i/6)$.}
\end{figure}
\bdm
Let $0<|q|<1$, and $a\in \mathbb C$, $a\ne 0, 1$. We are going to express the 
$q$-Jackson integral (\ref{2:1}) as the difference of the two infinite sums 
and apply the identity 
\begin{eqnarray} \label{2:2} 
&&\sum_{k=0}^M f(q^k){\mathscr D}_{q^{-1}} g(q^k)q^k=\frac{f(q^M)g(q^M)-f(q^{-1})
g(q^{-1})}{q^{-1}-1}\nonumber\\
&&\hspace{4.5cm}-\sum_{k=0}^M g(q^{k-1}){\mathscr D}_{q^{-1}} f(q^k) q^k.
\end{eqnarray}
Let $n\ge m$. Then, for one side since $w(q^{-1};a;q)=0$, and using 
the identities \cite[(14.24.7), (14.24.9)]{Koekoeketal},  
one has
\[ \begin{array}{l}
\displaystyle \sum_{k=0}^\infty U_m^{(a)}(q^k;q)U_n^{(a)}(q^k;q)
w(q^k;a;q)q^k\\[3mm]  
\hspace{0.5cm}\displaystyle = \frac{a(1-q)}{q^{2-n}}\lim_{M\to \infty}  
\sum_{k=0}^M\!{\mathscr D}_{q^{-1}}[w(q^k;a;q)U_{n-1}^{(a)}(q^k;q)]U_m^{(a)}
(q^k;q) q^k\\[5mm] 
\hspace{0.5cm}\displaystyle =
a q^{n-1}  
\lim_{M\to \infty} U_m^{(a)}(q^M;q) U_{n-1}^{(a)}(q^M;q)w(q^M;a;q) \\[3mm] 
\hspace{0.5cm}\hspace{0.6cm}\displaystyle
+aq^{n-1} (q^m-1)\! 
\lim_{M\to\infty}
\sum_{k=0}^{M-1} w
(q^k;a;q) U_{n-1}^{(a)}(q^k;q) U_{m-1}^{(a)}(q^k;q) q^k.
\end{array}\]
Following an analogous process as before, and since $w(aq^{-1};a;q)=0$, 
we have 
\begin{eqnarray*}
&&\displaystyle \sum_{k=0}^\infty U_m^{(a)}(aq^k;q)
U_n^{(a)}(aq^k;q)w(aq^k;a;q)a q^k \\[3mm]
&&\hspace{0.5cm}\displaystyle =a q^{n-1} \lim_{M\to \infty}  U_m^{(a)}(aq^M;q)U_{n-1}^{(a)}(aq^M;q)
 w(aq^M;a;q)\\[3mm]  
&&\hspace{0.2cm}\hspace{0.6cm}\displaystyle +
a q^{n-1} (q^m-1) 
\lim_{M\to\infty}
\sum_{k=0}^{M-1}  w(aq^k;a;q) U_{n-1}^{(a)}
(aq^k;q) U_{m-1}^{(a)}(aq^k;q) aq^k.
\end{eqnarray*}
Therefore, if  $m<n$, and since $m$ is finite one can first repeat the previous 
process $m+1$ times obtaining
\[\begin{array}{l}
\displaystyle \sum_{k=0}^\infty U_m^{(a)}(q^k;q)U_n^{(a)}(q^k;q)
w(q^k;a;q)q^k\\[4mm]  
\hspace{0.5cm}\displaystyle =\lim_{M\to \infty} \sum_{\nu=1}^{m+1} (-a q^n)^\nu q^{-\nu(\nu+1)/2} 
(q^{-m+\nu-1};q)_{\nu}\nonumber\\[5mm]
\hspace{3.0cm}\times U_{m-\nu+1}^{(a)}(q^M;q) U_{n-\nu}^{(a)}(q^M;q)w(q^M;a;q),
\end{array}\]
and 
\[ \begin{array}{l}
\displaystyle \sum_{k=0}^\infty U_m^{(a)}(aq^k;q)
U_n^{(a)}(aq^k;q)w(aq^k;a;q)a q^k \\[3mm]
\hspace{0.5cm}\displaystyle =\lim_{M\to \infty}\sum_{\nu=1}^{m+1} (-a q^n)^\nu q^{-\nu(\nu+1)/2} 
(q^{-m+\nu-1};q)_{\nu}\\[4mm]
\hspace{3cm}\times U_{m-\nu+1}^{(a)}(aq^M;q) U_{n-\nu}^{(a)}(aq^M;q)w(aq^M;a;q).
\end{array}\]
Hence since the difference of both limits, term by term, goes to 0 since $|q|<1$, then
\[
\displaystyle \int_a^1 U_n^{(a)}(x;q)U_m^{(a)}(x;q) (qx,qx/a;q)_\infty d_q x =0.
\]

For $n=m$, following the same idea, we have 
\[\begin{array}{l}
\displaystyle \int_a^1 U_n^{(a)}(x;q)U_n^{(a)}(x;q) w(x;a;q)d_q x\\[3mm]
\hspace{0.5cm}=\displaystyle\frac{a (q^n-1)}{q^{1-n}}\sum_{k=0}^\infty \Biggl( w(q^k;a;q)
\left(U_{n-1}^{(a)}(q^k;q)\right)^2 q^k\\[3mm]
\hspace{4cm}-a w(aq^k;a;q) \left(U_{n-1}^{(a)}
(aq^k;q)\right)^2 q^k\Biggr)\\[3mm] 
\hspace{0.5cm}=\displaystyle (-a)^n (q;q)_n q^{n 
\choose 2} \sum_{k=0}^\infty \left( w(q^k;a;q) q^k-a\  w(aq^k;a;q) 
q^k\right)\\
\hspace{0.5cm}=\displaystyle (-a)^n (q;q)_n (q;q)_\infty\, q^{n \choose 2}
\sum_{k=0}^\infty\left((q^{k+1}/a;q)_\infty-a(aq^{k+1};q)_\infty\right)\frac
{q^k}{(q;q)_k},
\end{array}\]
since it is known that in this case \cite[(14.24.2)]{Koekoeketal}
\begin{eqnarray*}
&&\int_a^1 U_n^{(a)}(x;q)U_n^{(a)}(x;q) w(x;a;q)d_q x\\
&&\hspace{4cm}=\displaystyle (-a)^n (q;q)_n (q;q)_\infty 
(a;q)_\infty (q/a;q)_\infty q^{n\choose 2}.
\end{eqnarray*}
Due to the normality of this polynomial sequence, i.e., $\deg U_n^{(a)}(x;q)=n$ 
for all $n\in \mathbb N_0$,  the uniqueness is straightforward, hence the
result holds.
\edm
From this result, and taking into account that the squared norm for the 
Al-Salam-Carlitz polynomials is known, we got the following consequence 
for which we could not find any reference.
\begin{corollary}
Let $a, q\in \mathbb C\setminus\{0\}$, $|q|<1$. Then
\[
\sum_{k=0}^\infty\left((q^{k+1}/a;q)_\infty-a(aq^{k+1};q)_\infty \right)\frac
{q^k}{(q;q)_k}=(a;q)_\infty (q/a;q)_\infty.
\]
\end{corollary}
The following case, which is just the Al-Salam-Carlitz polynomials for the $|q|>1$ case,
is commonly called the Al-Salam-Carlitz II polynomials. 
\begin{theorem}
Let $a,q\in\mathbb C$, $a\ne 0,1$, $|q|>1$. Then, the Al-Salam-Carlitz 
polynomials are unique (up to a multiplicative constant) satisfying 
the property of orthogonality given by
{\small \begin{eqnarray} \label{2:22}
&&\hspace{-0.65cm}\int_a^1 U_n^{(a)}(x;q^{-1})U_m^{(a)}(x;q^{-1}) 
(q^{-1}x;q^{-1})_\infty (q^{-1}x/a;q^{-1})_\infty d_{q^{-1}} x\nonumber\\[1mm]
&&\hspace{-0.4cm} =(-a)^n (1-q^{-1})(q^{-1};q^{-1})_n (q^{-1};q^{-1})_\infty
(a;q^{-1})_\infty (q^{-1}/a;q^{-1})_\infty \, q^{-{n\choose 2}}\delta_{m,n}. 
\end{eqnarray}}
\end{theorem}
\bdm
Let us denote $q^{-1}$ by $p$, then $0<|p|<1$. 
For $a\in\mathbb C$, $a\ne 0, 1$. Then, by using the identity (\ref{2:2}) 
replacing $q\mapsto p$, and taking into account that 
$w(aq;a;p)=w(q;a;p)=0$ and \cite[(14.24.9)]{Koekoeketal}, 
for $m<n$ one has
\[\begin{array}{l}
\displaystyle \sum_{k=0}^\infty a w(ap^k;a;p) U_m^{(a)}(ap^k;p)
U_n^{(a)}(ap^k;p)p^k \\[3mm] 
\hspace{0.5cm}= \displaystyle 
a p^{n-1}
\lim_{M\to \infty} 
U_m^{(a)}(ap^M;p)U_{n-1}^{(a)}(ap^M;p) w(ap^M;a;p)\\[3mm]
\hspace{0.5cm}\hspace{0.3cm}\displaystyle +a p^{n-1}(1-p^m) 
\lim_{M\to \infty} 
\sum_{k=0}^{M-1}  a w(ap^k;a;p) 
U_{n-1}^{(a)}(ap^k;p) U_{m-1}^{(a)}(ap^k;p)p^k.
\end{array}\]
Following the same idea from the previous result, we have
\[\begin{array}{l}
\displaystyle \sum_{k=0}^\infty  w(p^k;a;p) U_m^{(a)}(p^k;p)
U_n^{(a)}(p^k;p)p^k \\[3mm] 
\hspace{0.5cm}\displaystyle =
a p^{n-1}
\lim_{M\to \infty} 
U_m^{(a)}(p^M;p)U_{n-1}^{(a)}(p^M;p) w(p^M;a;p)\\[3mm]
\hspace{0.5cm}\hspace{0.3cm}\displaystyle +a p^{n-1}(1-p^m) 
\lim_{M\to \infty} 
\sum_{k=0}^{M-1}   w(p^k;a;p) 
U_{n-1}^{(a)}(p^k;p) U_{m-1}^{(a)}(p^k;p)p^k.
\end{array}\]
Therefore, the property of orthogonality holds for $m<n$. 
Next, if $n=m$, we have
\[\begin{array}{l}
\displaystyle
\int_a^1 U^{(a)}_n(x;p)U^{(a)}_n(x;p)  w(x;a;p)\, d_p x
\\
\hspace{0.5cm}=\displaystyle\frac{a(p^n-1)}{p^{1-n}}\!\!\sum_{k=0}^\infty\Biggl( a w
(ap^k;a;p) \left(U_{n-1}^{(a)}(ap^k;p)\right)^2 p^{k}\\[4mm]
\hspace{4cm}- w(p^{k};a;p)\left(
U_{n-1}^{(a)}(p^{k};p)\right)^2 p^{k}\Biggr)\\[3mm]
\hspace{0.5cm}=\displaystyle (-a)^n 
(p;p)_n p^{n\choose 2} \left(\sum_{k=0}^\infty a w(ap^{k};a;p) 
p^{k}- w(p^{k};a;p) p^{k}\right)\\[3mm]
\hspace{0.5cm}=\displaystyle (-a)^n \
(q^{-1};q^{-1})_n (p;p)_\infty p^{n\choose 2} 
\sum_{k=0}^\infty 
\frac{q^k\left(a(p^{k+1}a;p)_\infty-(p^{k+1}/a;p)_\infty\right)}
{(p;p)_k}\\[3mm]
\hspace{0.5cm}=\displaystyle (-a)^n (q^{-1};q^{-1})_n (p;p)_\infty 
(a;p)_\infty(p/a;p)_\infty p^{n\choose 2}.
\end{array}\]
Using the same argument as in Theorem \ref{thm:3.1}, the uniqueness
holds, so the claim follows.
\edm
\begin{remark} Observe that in the previous theorems if $a=q^m$, with 
$m\in \mathbb Z$, $a\ne 0$, after some logical cancellations, the set of points 
where we need to calculate the $q$-integral is easy to compute. 
For example,  if $0<aq<1$ and $0<q<1$, one obtains the sum 
\cite[p. 537, (14.25.2)]{Koekoeketal}.
\end{remark}
\begin{remark} 
The $a=1$ case is special because it is not  considered in the 
literature. In fact, the linear form associated with 
the Al-Salam-Carlitz polynomials $\bf u$ is quasi-definite and 
fulfills  the Pearson-type distributional equations
\[
{\mathscr D}_q[(x-1)^2 {\bf u}]=\frac {x-2}{1-q}  {\bf u}\quad
\text{and} \quad {\mathscr D}_{q^{-1}}[q^{-1}{\bf u}]=\frac {x-2}{1-q}  {\bf u}.
\]
Moreover, the Al-Salam-Carlitz polynomials fulfill the  three-term recurrence 
relation  \cite[(14.24.3)]{Koekoeketal}
\begin{equation}\label{2:5}
xU^{(a)}_{n}(x;q)=U^{(a)}_{n+1}(x;q)+(a+1)q^nU^{(a)}_{n}(x;q)-
aq^{n-1}(1-q^n)U^{(a)}_{n-1}(x;q),
\end{equation}
where $n=0, 1, \dots,$ with initial conditions $U^{(a)}_{0}(x;q)=1$, 
$U^{(a)}_{1}(x;q)=x-a-1$. 

Therefore, we believe that it will be interesting to study such a case for its 
peculiarity because the coefficient $q^{n-1}(1-q^n)\ne 0$  for all $n$, 
so one can apply Favard's result.
\end{remark}
\subsection{The $|q|=1$ case.}

\noindent In this section we only consider the case where $q$ is a root of unity.
Let $N$ be a positive integer such that $q^N=1$ then, due to the recurrence 
relation \eqref{2:5} and following the same idea that the authors did in
\cite[Section 4.2]{cola2}, we apply the following process:
\begin{enumerate}
\item The sequence $(U_n^{(a)}(x;q))_{n=0}^{N-1}$ is
orthogonal with respect to the Gaussian quadrature
\[
\langle {\bf v},p\rangle:=\sum_{s=1}^{N} \gamma_1^{(a)}\dots \gamma_{N-1}^{(a)}
\frac {p(x_s)}{\left(U^{(a)}_{N-1}(x_s)\right)^2},
\]
where $\{x_1,x_2,\dots,x_N\}$ are the zeros of $U_N^{(a)}(x;q)$ for such 
value of $q$.
\item Since $\langle {\bf v}, U_n^{(a)}(x;q)U_n^{(a)}(x;q)\rangle=0$, 
we need to modify such a linear form. 

Next, we can prove that the sequence $(U_n^{(a)}(x;q))_{n=0}^{2N-1}$ is 
orthogonal with respect to the bilinear form
\[
\langle p, r \rangle_{2}=\langle {\bf v},pq\rangle+
\langle {\bf v},{\mathscr D}^N_q p {\mathscr D}^N_q r\rangle,
\]
since ${\mathscr D}_q U_n^{(a)}(x;q)=(q^n-1)/(q-1)U_{n-1}^{(a)}(x;q)$.
\item Since  $\langle U_{2N}^{(a)}(x;q), U_{2N}^{(a)}(x;q) \rangle_{2}=0$ and taking into account 
what we did before, we consider the linear form
\[
\langle p, r \rangle_{3}=\langle {\bf v},pq\rangle+
\langle {\bf v},{\mathscr D}^N_q p {\mathscr D}^N_q r\rangle+
\langle {\bf v},{\mathscr D}^{2N}_q p {\mathscr D}^{2N}_q r\rangle.
\]
\item Therefore one can obtain a sequence of bilinear forms such that the 
Al-Salam-Carlitz polynomials are orthogonal with respect to them.
\end{enumerate}
\section{A generalized generating function for Al-Salam-Carlitz polynomials}

\noindent For this section, we are going to assume $|q|>1$, or $0<|p|<1$. Indeed, 
by starting with the generating functions for Al-Salam-Carlitz polynomials 
\cite[(14.25.11-12)]{Koekoeketal}, we derive generalizations using the 
connection relation for these polynomials.
\begin{theorem}
Let $a, b, p\in\mathbb C\setminus \{0\}$, $|p|<1$, $a, b\ne 1$. Then
\begin{equation}
\label{con1}
U_n^{(a)}(x;p)=(-1)^n(p;p)_np^{-{n \choose 2}}\sum_{k=0}^{n}
\frac{(-1)^ka^{n-k}(b/a;p)_{n-k}p^{\binom{k}{2}}}{(p;p)_{n-k}(p;p)_k}
U_k^{(b)}(x;p).
\end{equation}
\end{theorem}
\begin{proof}
If we consider the generating function for Al-Salam-Carlitz polynomials 
\cite[(14.25.11)]{Koekoeketal}
\[
\frac{(xt;p)_{\infty}}{(t,a t;p)_{\infty}}=
\sum_{n=0}^{\infty}\frac{(-1)^np^{n \choose 2}}{(p;p)_n}U_n^{(a)}(x;p)t^n,
\]
and multiply both sides by ${(b t;p)_{\infty}}/{(b t;p)_{\infty}}$,
obtaining
\begin{equation}
\label{con2}
\sum_{n=0}^{\infty}\frac{(-1)^np^{n \choose 2}}{(p;p)_n}U_n^{(a)}
(x;p)t^n=\frac{(bt;p)_\infty}{(at;p)_\infty}\sum_{n=0}^{\infty}
\frac{(-1)^np^{n \choose 2}}{(p;p)_n}U_n^{(b)}(x;p)t^n.
\end{equation}
If we now apply the $q$-binomial theorem \cite[(1.11.1)]{Koekoeketal}
\[
\frac{(az;p)_\infty}{(z;p)_\infty}=\sum_{k=0}^\infty \frac{(ap;p)_n}
{(p;p)_n}z^n,\quad 0<|p|<1, \quad |z|<1,
\]
to (\ref{con2}), and then collect powers of $t$, we obtain
\begin{eqnarray*}
&&\sum_{k=0}^{\infty}t^k\sum_{m=0}^{k}\frac{(-1)^ma^{k-m}(b/a;p)_{k-m}
p^{m \choose 2}}{(p;p)_{k-m}(p;p)_m}U_m^{(b)}(x;p)\\[-0mm]
&&\hspace{3cm}=\sum_{n=0}^{\infty}\frac{(-1)^n
p^{n \choose 2}}{(p;p)_n}U_n^{(a)}(x;p)t^n.
\end{eqnarray*}
Taking into account this expression, the result follows.
\end{proof}
\begin{theorem} 
Let $a, b, p\in \mathbb C\setminus\{0\}$, $|p|<1$, $a,b\ne 1$, $t\in\mathbb C$, 
$|at|<1$. Then
\begin{equation}
\label{ASCGenfun}
(at;p)_{\infty}\,{}_1\phi_1\left(\begin{array}{c} x \\ at \end{array}; p,t
\right)=\sum_{k=0}^{\infty}\frac{p^{k(k-1)}}{(p;p)_k}\,{}_1\phi_1\left(
\begin{array}{c} b/a \\  0 \end{array};p, atp^k\right)U_k^{(b)}(x;p)t^k,
\end{equation}
where 
\begin{eqnarray*}
&&\hspace{-0.2cm}{}_r\phi_s\left(\begin{array}{c} a_1, a_2, \dots,  a_r\\ b_1, b_2, \dots,b_s \end{array}; p,z\right)\\
&&\hspace{1cm}=
\sum_{k=0}^\infty \frac{(a_1;p)_k(a_2;p)_k\cdots (a_r;p)_k}
{(b_1;p)_k(b_2;p)_k\cdots (b_s;p)_k}\frac{z^k}{(p;p)_k}(-1)^{(1+s-r)k}p^{(1+s-r) {k\choose 2}},
\end{eqnarray*}
is the unilateral basic hypergeometric series.
\end{theorem}
\begin{proof}
We start with a generating function for Al-Salam-Carlitz polynomials
\cite[(14.25.12)]{Koekoeketal} 
\[
(a t;q)_{\infty}\,{}_1\phi_1\left(\begin{array}{c} x \\ at \end{array}; 
q,t\right)=\sum_{k=0}^\infty \frac{q^{n(n-1)}}{(q;q)_n} V^{(a)}_n(x;q)t^n
\]
and (\ref{con1}) to obtain
\begin{eqnarray*}
&&(at;p)_{\infty}\,{}_1\phi_1\left(\begin{array}{c} x \\ at \end{array}; p,t
\right)\\
&&\hspace{1.8cm}=\sum_{n=0}^{\infty}t^n(-1)^n p^{n \choose 2}\sum_{k=0}^n\frac{(-1)^k
a^{n-k}(b/a;p)_{n-k}p^{k \choose 2}}{(p;p)_{n-k}(p;p)_k}U_k^{(b)}(x;p).
\end{eqnarray*}
If we reverse the order of summations, shift the $n$ variable by a factor 
of $k$, using the basic properties of the $q$-Pochhammer symbol, and 
\cite[(1.10.1)]{Koekoeketal}. Observe that we can reverse the order of summation since 
our sum is of the form
\[
\sum_{n=0}^\infty a_n \sum_{k=0}^n c_{n,k} U_k^{(a)}(x;p),
\]
where 
\[
a_n=t^n,\qquad 
c_{n,k}=\frac{(-1)^ka^{n-k}(b/a;p)_{n-k}p^{\binom{k}{2}}}{(p;p)_{n-k}(p;p)_k}.
\]
In this case, one has
\[
|a_n|\le |t|^n,\quad |c_{n,k}|\le K(1+n)^{\sigma_1} |a|^n,  
\]
and $|U_n^{(a)}(x;p)|\le (1+n)^{\sigma_2}$, where $K_1$, $\sigma_1$,
and $\sigma_2$  are positive constants independent of $n$. Therefore, 
if $|at|<1$, then
\[
\left|\sum_{n=0}^\infty a_n \sum_{k=0}^n c_{n,k} U_k^{(a)}(x;p)\right|<\infty,
\]
and this completes the proof.
\end{proof}
As we saw in Section \ref{sectionorthog}, the orthogonality 
relation for Al-Salam-Carlitz 
polynomials for $|q|>1$, $|p|<1$, and $a\ne 0, 1$ is 
\[
\int_{\Gamma} U_n^{(a)}(x;p)U_m^{(a)}(x;p) 
w(x;a;p) d_{p} x=d_n^2 \delta_{n,m}.
\]
Taking this result in mind, 
the following result follows.
\begin{theorem} 
Let $a, b, p\in \mathbb C\setminus\{0\}$, $t\in \mathbb C$, $|at|<1$, $|p|<1$, 
$m\in \mathbb N_0$. Then
\[\begin{array}{rl}
\displaystyle \int_a^1 {}_1\phi_1\left(\begin{array}{c} q^{-x} \\ at \end{array}; q,t\right) 
U_m^{(b)}(q^{-x};p)(q^{-1}x;q^{-1})_\infty(q^{-1}x/a;q^{-1})_\infty dq^{-1}\\[4mm]
= \displaystyle 
\big(-bt\big)^mq^{3 {m\choose 2}}(b;p)_\infty (p/b;p)_{\infty}\,
{}_1\phi_1\left(\begin{array}{c}  b/a \\ 0 \end{array};q, atq^m\right).
\end{array}\]
\end{theorem} 
\begin{proof}
From (\ref{ASCGenfun}), we replace $x\mapsto p^{x}$ and multiply both sides 
by $U_m^{(b)}(x;p)w(x;a;p)$, and by using the 
orthogonality relation (\ref{2:22}), the desired result holds.
\end{proof}

Note that the application of connection relations to the rest of the known generating
functions for Al-Salam-Carlitz polynomials \cite[(14.24.11), (14.25.11)]{Koekoeketal}
leave these generating functions invariant.

\section*{Acknowledgments}

\noindent The author R. S. Costas-Santos acknowledges financial support by National Institute of 
Standards and Technology. The authors thank the anonymous referee for her/his valuable 
comments and suggestions. They contributed to improve the presentation of the manuscript.


\def\cprime{$'$} \def\dbar{\leavevmode\hbox to 0pt{\hskip.2ex \accent"16\hss}d}


\end{document}